\documentclass[12pt, reqno]{amsart}
\usepackage{amssymb, amsmath, amsthm, cancel, hyperref, color, enumerate, verbatim, mathrsfs, mathtools, gensymb, caption, subcaption, graphicx, multicol, setspace, enumitem}
\usepackage[title]{appendix}

\usepackage[margin=1in]{geometry}

\newtheorem{thm}{Theorem}[section]
\newtheorem{lem}[thm]{Lemma}
\newtheorem{prop}[thm]{Proposition}

\theoremstyle{definition}

\theoremstyle{remark}

\newcommand{\NN}{\mathbb{N}}
\newcommand{\ZZ}{\mathbb{Z}}

\newcommand{\SL}{\text{SL}_2(\ZZ)}
\newcommand{\mat}[4]{\begin{pmatrix} #1 & #2 \\ #3 & #4 \end{pmatrix}}

\title{Linking congruences for PED and POD partitions}

\author{Dalen Dockery}
\address{Department of Mathematics, University of Tennessee, Knoxville, TN 37996, USA}
\email{ddocker5@vols.utk.edu}

\author{Marie Jameson}
\address{Department of Mathematics, University of Tennessee, Knoxville, TN 37996, USA}
\email{mjameso2@utk.edu}

\begin{document}

\begin{abstract}
Recent work of Garvan, Sellers, Smoot, and others has made connections between infinite
families of congruences for various partition functions. Here, we apply this approach to families of congruences for PED and POD partitions and find that they are naturally linked to congruence families for overpartitions into odd parts and overpartitions. 
\end{abstract}

\maketitle

\section{Introduction}

A recent development in the theory of partition congruences is the notion of congruence multiplicity, which occurs when divisibility properties of a single family of modular functions manifest as simultaneous congruence families for multiple partition functions. This phenomenon was first observed by Garvan, Sellers, and Smoot (\cite{GSS}) during their study of generalized Frobenius partitions, in which they established the logical equivalence of the congruence families
\[c\phi_2(n) \equiv 0 \pmod{5^\alpha}, \quad c\psi_2(m) \equiv 0 \pmod{5^\alpha}\]
for all $\alpha \geq 1$, where $m,n \geq 0$ satisfy $12n \equiv 1 \pmod{5^\alpha}$ and $-6m \equiv 1 \pmod{5^\alpha}.$ Garvan, Sellers, and Smoot linked these congruences by first defining sequences of modular functions that enumerate $c\phi_2(n)$ and $c\psi_2(m)$ in the arithmetic progressions of interest and subsequently mapping one sequence to the other via a conjugated Atkin-Lehner involution. As exemplified their work, congruence multiplicity is a very alluring method to prove new congruence families, as it allows one to bypass long, tedious direct proofs and instead connect to congruences existing in the literature.

Sellers and Smoot (\cite{SellersSmoot}) utilized this technique to connect congruences families for PEND and POND partitions, which are combinatorial objects that enumerate the number of partitions in which the even and odd parts are not distinct, respectively. More specifically, they studied the congruences
\begin{align}
\mathrm{pend}\left(3^{2\alpha+1}n+\frac{17\cdot 3^{2\alpha}-1}{8}\right) &\equiv 0 \pmod{3}, \label{eq:pendcong} \\
\mathrm{pond}\left(3^{2\alpha+1}n + \frac{23 \cdot 3^{2\alpha}+1}{8}\right) &\equiv 0 \pmod{6}, \label{eq:pondcong} 
\end{align}
where $\alpha \geq 1$ and $n \geq 0$. Here $\mathrm{pend}(n)$ and $\mathrm{pond}(n)$ count the number of PEND and POND partitions of $n$; these congruences are due to Sellers (\cite{Sellers24}). 

In light of their similar combinatorial definitions to PEND and POND partitions, Sellers and Smoot suggested that PED and POD partitions may also be connected in this way. Note that a PED (resp. POD) partition is one in which the even (resp. odd) parts are distinct, and these partitions are counted by the functions $\mathrm{ped}(n)$ and $\mathrm{pod}(n)$. As further evidence for the connection between PED and POD partitions, Sellers and Smoot noted that they satisfy congruences that bear a striking resemblance to \eqref{eq:pendcong}-\eqref{eq:pondcong}. In particular, Andrews, Hirschhorn, and Sellers (\cite{HSpod,AHSped}) proved that
\begin{align}
\mathrm{ped}\left(3^{2\alpha+1}n+\frac{17\cdot 3^{2\alpha}-1}{8}\right) &\equiv 0 \pmod{6}, \label{eq:pedcong} \\
\mathrm{pod}\left(3^{2\alpha+1}n + \frac{23 \cdot 3^{2\alpha}+1}{8}\right) &\equiv 0 \pmod{3}, \label{eq:podcong}
\end{align}
for all $\alpha\geq 1$ and $n\geq 0.$

Our goal in this work is to address this question of whether or not PED and POD partitions are linked via congruence multiplicity. Perhaps surprisingly, we find that the technique pioneered by Garvan, Sellers, Smoot (and others) does not link these partition congruences to each other. Instead, we show that POD partitions are linked to overpartitions, and similarly we show that PED partitions are linked to overpartitions with odd parts. Recall that an overpartition is a partition in which the first occurrence of each part size may be overlined, and to set notation, let $\overline{p}(n)$ (resp. $\overline{p_o}(n)$) count the number of overpartitions of $n$ (resp. overpartitions of $n$ with odd parts). Hirschhorn and Sellers (\cite{HSopnodd,HSoverpartitions}) established the congruences
\begin{align}
\overline{p_o} \left(3^{2\alpha}(3n+1)\right) \equiv 0 \pmod{6}, \label{eq:pobarcong} \\
\overline{p} \left( 3^{2\alpha}(3n+2)\right) \equiv 0\pmod{12}, \label{eq:pbarcong}
\end{align}
respectively, for all $\alpha\geq 1$ and $n\geq 0.$

In order to link the congruences \eqref{eq:pedcong}-\eqref{eq:pbarcong}, we first describe them using families of modular functions $L_\alpha^{\mathrm{ped}}, L_\alpha^{\mathrm{pod}},
L_\alpha^{\overline{p_o}},$ and $L_\alpha^{\overline{p}}$ (which are defined explicitly in Section \ref{sec:inf_fams}). Our first main theorem connects these modular functions using a conjugated Atkin-Lehner operator $\gamma$ (which is defined explicitly in Section \ref{sec:proof}).
\begin{thm} \label{thm:congfams}
For all $\alpha\geq 1,$ we have
\begin{align*}
L_\alpha^{\mathrm{ped}} \mid \gamma &= (-1)^\alpha \cdot L_\alpha^{\overline{p_o}}\\
L_\alpha^{\mathrm{pod}} \mid \gamma &= (-1)^{\alpha+1}\frac{1}{4} \cdot L_\alpha^{\overline{p}}.
\end{align*}
\end{thm}

To further strengthen the connection between these partition functions, we also note that they satisfy related internal congruences. For instance, we have
\begin{align*}
\mathrm{ped}(3n+1) &\equiv -\mathrm{ped}(27n+10) \hspace{-10em} &&\pmod{6}\\
\overline{p_o}(3n) &\equiv \overline{p_o}(27n) &&\pmod{12}
\end{align*}
for all $n\geq 0$ (and similar internal congruences for $\mathrm{pod}(n)$ and $\overline{p}(n)$ as described in Section \ref{sec:internal}). (In fact, these internal congruences could be used as a tool to prove equations \eqref{eq:pedcong}-\eqref{eq:pbarcong} using induction.) Then after defining appropriate modular functions $L^{\mathrm{ped}}, L^{\mathrm{pod}},
L^{\overline{p_o}},$ and $L^{\overline{p}}$ and operator $\gamma,$ we have the following second main theorem, which connects their internal congruences as well.

\begin{thm}\label{thm:main2}
For all $\alpha\geq 1,$ we have
\begin{align*}
L^{\mathrm{ped}} \mid \gamma &= \frac{1}{2} L^{\overline{p_o}}\\
L^{\mathrm{pod}} \mid \gamma &= -\frac{1}{8} L^{\overline{p}}.
\end{align*}
\end{thm}

The remainder of this paper is organized as follows. Section \ref{sec:background} collects the necessary preliminary results that we require on modular functions. In Section \ref{sec:inf_fams}, we define the sequences of modular functions $L_\alpha^*$ necessary for the proof of Theorem \ref{thm:congfams}, which we carry out in Section \ref{sec:proof}. Finally, in Section \ref{sec:internal}, we prove Theorem \ref{thm:main2} by calculating the appropriate generating functions, defining the modular functions $L^*,$ and applying the appropriate operator $\gamma$ to complete the proof.


\section{Background}\label{sec:background}
In this work, we assume familiarity with the theory of modular functions, though we briefly summarize below the crucial facts that we shall need concerning Atkin's $U_d$-operator and the Dedekind eta-function. For a more detailed description on these topics, see e.g. \cite{DS}.

Let $\mathcal{M}(\Gamma_0(N))$ denote the $\mathbb{C}$-vector space of all weakly holomorphic modular functions of level $\Gamma_0(N).$ For integers $d \geq 1$, Atkin's $U_d$-operator is defined on $\mathcal{M}(\Gamma_0(N))$ by its action on Fourier series:
\[U_d \left(\sum\limits_{n \geq n_0} a(n) q^n \right) \coloneqq \sum\limits_{dn \geq n_0} a(dn) q^n,\]
where $q=e^{2\pi i\tau}.$

The following two lemmas summarize important facts about the $U_d$-operator. 

\begin{lem}[Lemmas 6 and 7 of \cite{AtkinLehner}]\label{lem:Umap}
Suppose $d, N$ are positive integers such that $d \mid N$. Then 
\[U_d : \mathcal{M}(\Gamma_0(N)) \rightarrow \mathcal{M}(\Gamma_0(N)).\]
Moreover, if $d^2 \mid N$ then 
\[U_d : \mathcal{M}(\Gamma_0(N)) \rightarrow \mathcal{M}\left(\Gamma_0\left(\frac{N}{d}\right)\right).\]
\end{lem}

\begin{lem}\label{lem:U_PS}
Let $F(q)$ and $G(q)$ be $q$-series and take $d \in \NN.$ Then
\[U_d\left(F(q^d)G(q)\right) = F(q) \cdot U_d\left(G(q)\right).\]
\end{lem}

Most of the modular functions of interest here arise from Dedekind's eta-function, given by
\[\eta(\tau) \coloneqq q^{1/24} \prod\limits_{n=1}^\infty (1-q^n).\]
More specifically, an eta-quotient is a function of the form 
\[\prod\limits_{\delta \mid N} \eta(\delta \tau)^{r_\delta}\]
for integers $N, r_\delta,$ and the following well-known results of Newman and Ligozat allow one to easily check the modularity of a given eta-quotient, as well as its order of vanishing at cusps. 

\begin{lem}[Theorem 1.64 of \cite{OnoWeb}]\label{lem:eta-quotient1}
Suppose that the eta-quotient $\prod\limits_{\delta \mid N} \eta(\delta \tau)^{r_\delta}$ satisfies
\begin{enumerate}
\item $\sum\limits_{\delta \mid N} r_{\delta} = 0,$
\item $\sum\limits_{\delta \mid N} \delta r_\delta \equiv 0 \pmod{24},$
\item $\sum\limits_{\delta \mid N} \frac{N}{\delta} r_\delta \equiv 0 \pmod{24},$
\item $\prod\limits_{\delta \mid N} \delta^{r_\delta}$  is a perfect square in $\mathbb{Q}$. 
\end{enumerate}
Then $\prod\limits_{\delta \mid N} \eta(\delta \tau)^{r_\delta} \in \mathcal{M}(\Gamma_0(N)).$ 
\end{lem}

\begin{lem}[Theorem 1.65 of \cite{OnoWeb}]\label{lem:eta-quotient2}
The order of vanishing of the eta-quotient 
$\prod\limits_{\delta \mid N} \eta(\delta \tau)^{r_\delta} \in \mathcal{M}(\Gamma_0(N))$ at the cusp $[a/c]_N$ of $\Gamma_0(N)$ is given by
\[\frac{N}{24} \sum\limits_{\delta \mid N} \frac{\gcd(c,\delta)^2r_\delta}{\gcd(c,\frac{N}{c})c\delta}.\]
\end{lem}

Moreover, the Dedekind eta-function satisfies a remarkable transformation law with respect to the action of $\SL.$

\begin{lem}[Chapter 3 of \cite{Apostol}]\label{lem:eta_trans}
For all $\mat{a}{b}{c}{d} \in \SL,$ 
\[\eta\left(\frac{a\tau+b}{c\tau+d}\right) = (-i(c\tau+d))^{1/2} \cdot \varepsilon(a,b,c,d) \cdot \eta(\tau).\]
Here, $\varepsilon(a,b,c,d)$ is the 24th root of unity defined by
\[\varepsilon(a,b,c,d) \coloneqq \begin{cases} \exp\left(\frac{b\pi i}{12}\right), & c=0, d=1, \\ \exp\left(\pi i\left(\frac{a+d}{12c} - s(d,c) - \frac{1}{4}\right)\right), & c>0, \end{cases}\]
where $s(\cdot,\cdot)$ is the usual Dedekind sum.
\end{lem}

Finally, throughout we will utilize the notation 
\[f_r \coloneqq \prod\limits_{n=1}^\infty (1-q^{rn})\]
for convenience. 

\section{Connecting infinite congruence families}\label{sec:inf_fams}
We begin by defining the sequences of modular functions that will enumerate these partition functions of interest, which we note have generating functions given by

\begin{align*}
\sum\limits_{n=0}^\infty \mathrm{ped}(n) q^n &= \frac{f_4}{f_1}, \\
\sum\limits_{n=0}^\infty \mathrm{pod}(n) q^n &= \frac{f_2}{f_1f_4}, \\
\sum\limits_{n=0}^\infty \overline{p_o}(n) q^n &= \frac{f_2^3}{f_1^2 f_4}, \\
\sum\limits_{n=0}^\infty \overline{p}(n) q^n &= \frac{f_2}{f_1^2}.
\end{align*}

For $\alpha \geq 1,$ let

\begin{align*}
P_\alpha^{\mathrm{ped}} &\coloneqq \frac{\eta(6 \cdot 3^{2\alpha+1} \tau)^4}{ \eta( 3^{2\alpha+1} \tau)^2  \eta(3 \cdot 3^{2\alpha+1} \tau)  \eta(12 \cdot 3^{2\alpha+1} \tau) } \cdot \frac{\eta(4\tau)}{\eta(\tau)}, \\
P_\alpha^{\mathrm{pod}} &\coloneqq \frac{\eta(2 \cdot 3^{2\alpha+1}\tau)^4 \eta(4 \cdot 3^{2\alpha+1}\tau)^6}{\eta(3^{2\alpha+1}\tau)^4 \eta(3 \cdot 3^{2\alpha+1}\tau)^3 \eta(6 \cdot 3^{2\alpha+1}\tau) \eta(12 \cdot 3^{2\alpha+1}\tau)} \cdot \frac{\eta(2\tau)}{\eta(\tau)\eta(4\tau)},\\
P_\alpha^{\overline{p_o}} &\coloneqq \frac{\eta(6 \cdot 3^{2\alpha+1} \tau)\eta(12 \cdot 3^{2\alpha+1} \tau)}{\eta(3^{2\alpha+1} \tau)^2} \cdot \frac{\eta(2\tau)^3}{\eta(\tau)^2\eta(4\tau)}, \\
P_\alpha^{\overline{p}} &\coloneqq \frac{\eta(2\cdot 3^{2\alpha+1} \tau)^{22} \eta(12\cdot 3^{2\alpha+1} \tau)}{\eta(3^{2\alpha+1} \tau)^{10} \eta(3\cdot 3^{2\alpha+1} \tau)^2 \eta(4\cdot 3^{2\alpha+1} \tau)^6 \eta(6\cdot 3^{2\alpha+1} \tau)^4} \cdot \frac{\eta(4\tau)}{\eta(2\tau)^2}.
\end{align*}

Using Lemmas \ref{lem:eta-quotient1} it is straightforward to verify that $P_\alpha^{*}\in \mathcal{M}(\Gamma_0(12 \cdot 3^{2\alpha+1}))$ for $* \in \{\mathrm{ped},\mathrm{pod},\overline{p_o},\overline{p}\}.$ Hence, by Lemma \ref{lem:Umap} we conclude
\[L_\alpha^{*} \coloneqq U_3^{2\alpha+1} \left(P_\alpha^{*}\right) \in \mathcal{M}(\Gamma_0(12)).\]
Routine calculations reveal 
\begin{align*}
L_\alpha^{\mathrm{ped}} &= \frac{f_6^4}{f_1^2 f_3 f_{12}} \sum\limits_{n=0}^\infty \mathrm{ped}\left(3^{2\alpha+1}n+\frac{17 \cdot 3^{2\alpha}-1}{8}\right) q^{n+1}, \\
L_\alpha^{\mathrm{pod}} &= \frac{f_2^4 f_4^6}{f_1^4 f_3^3 f_6 f_{12}} \sum\limits_{n=0}^\infty \mathrm{pod}\left(3^{2\alpha+1}n+\frac{23 \cdot 3^{2\alpha}+1}{8}\right) q^{n+1}, \\
L_\alpha^{\overline{p_o}} &= \frac{f_6f_{12}}{f_1^2} \sum\limits_{n=0}^\infty \overline{p_o} \left(3^{2\alpha+1}n + 3^{2\alpha}\right) q^{n+1}, \\
L_\alpha^{\overline{p}} &= \frac{f_2^{22} f_{12}}{f_1^{10} f_3^2 f_4^6 f_6^4} \sum_{n=0}^\infty \overline{p}\left( 3^{2\alpha+1}n+2\cdot 3^{2\alpha} \right)q^{2n+1}.
\end{align*}

For example, using Lemma \ref{lem:U_PS} we have
\begin{align*}
L_\alpha^{\mathrm{ped}} &= U_3^{2\alpha+1} \left( q^{\frac{7 \cdot 3^{2\alpha}+1}{8}} \frac{f_{6\cdot 3^{2\alpha+1}}^4}{f_{3^{2\alpha+1}}^2 f_{3\cdot 3^{2\alpha+1}} f_{12\cdot 3^{2\alpha+1}}} \cdot \sum\limits_{n=0}^\infty \mathrm{ped}(n) q^n\right) \\
&= \frac{f_6^4}{f_1^2 f_3 f_{12}} U_3^{2\alpha+1} \left(\sum\limits_{n \geq \frac{7 \cdot 3^{2\alpha}+1}{8}} \mathrm{ped}\left(n - \frac{7 \cdot 3^{2\alpha}+1}{8}\right) q^n\right) \\
&= \frac{f_6^4}{f_1^2 f_3 f_{12}} \sum_{n=1}^\infty \mathrm{ped}\left(3^{2\alpha+1}n - \frac{7 \cdot 3^{2\alpha}+1}{8}\right) q^{n}\\
&= \frac{f_6^4}{f_1^2 f_3 f_{12}} \sum\limits_{n=0}^\infty \mathrm{ped}\left(3^{2\alpha+1}n + \frac{17 \cdot 3^{2\alpha} - 1}{8}\right) q^{n+1},
\end{align*}
as desired.

Note that we have constructed these modular functions so that the congruences \eqref{eq:pedcong}-\eqref{eq:pbarcong} are equivalent to the statements that
\begin{align*}
L_\alpha^{\mathrm{ped}} &\equiv 0\pmod{6}\\
L_\alpha^{\mathrm{pod}} &\equiv 0\pmod{3}\\
L_\alpha^{\overline{p_o}} &\equiv 0\pmod{6}\\
L_\alpha^{\overline{p}} &\equiv 0\pmod{12}.
\end{align*}

\section{Proof of Theorem \ref{thm:congfams}} \label{sec:proof}

Following recent work \cite{GSS, SellersSmoot} of Garvan, Sellers, and Smoot, we first define a so-called ``conjugation operator'' obtained by conjugating an Atkin-Lehner operator with the map sending $q\mapsto -q.$ More precisely, we consider
\[\nu:= \begin{pmatrix}1&1/2\\0&1\end{pmatrix}\qquad \text{and}\qquad W := \begin{pmatrix}4&-1\\12\cdot 3^{2\alpha+1}&1-3^{2\alpha+2}\end{pmatrix}\] (noting that $W$ is the well-known Atkin-Lehner operator for modular forms of level $N=12\cdot 3^{2\alpha+1}$ with prime power $q=4$ \cite[Section 2.4]{OnoWeb}) and use them to define
\[\nu W \nu = \begin{pmatrix}4+6\cdot 3^{2\alpha+1}& \frac{3+3^{2\alpha+2}}{2}\\ 12\cdot 3^{2\alpha+1}& 1 + 3^{2\alpha+2}\end{pmatrix},\]
which reduces to 
\[\gamma = \begin{pmatrix}2+3\cdot 3^{2\alpha+1}& \frac{3+3^{2\alpha+2}}{4}\\ 6\cdot 3^{2\alpha+1}& \frac{1 + 3^{2\alpha+2}}{2}\end{pmatrix}.\]

Then, in order to compute $P_\alpha^{\mathrm{ped}} \mid \gamma$ and $P_\alpha^{\mathrm{pod}} \mid \gamma,$ it suffices to compute $\eta(\delta \gamma \tau)$ for the appropriate positive integers $\delta.$ However, in order to apply Lemma \ref{lem:eta_trans}, we must factor the corresponding matrix via
\[\mat{A}{B}{C}{D} = \mat{A/g}{-y}{C/g}{x} \mat{g}{Bx+Dy}{0}{(AD-BC)/g}\]
(where $g:=\gcd(A,C)$ and $x,y\in \mathbb{Z}$ satisfy $Ax+By=g$) in order to obtain a matrix in $\SL.$

For example, if $\delta=4$ then we have
\begin{align*}
4\gamma(\tau) &= \mat{8+12\cdot 3^{2\alpha+1}}{3+3^{2\alpha+2}}{6\cdot 3^{2\alpha+1}}{\frac{1+3^{2\alpha+2}}{2}}(\tau) = \mat{4+6\cdot 3^{2\alpha+1}}{\frac{-1-3^{2\alpha+2}}{2}}{3^{2\alpha+2}}{\frac{1-3^{2\alpha+2}}{4}}\mat{2}{1}{0}{2}(\tau)\\
&=\mat{4+6\cdot 3^{2\alpha+1}}{\frac{-1-3^{2\alpha+2}}{2}}{3^{2\alpha+2}}{\frac{1-3^{2\alpha+2}}{4}}(\tau+1/2)\\
\end{align*}
and so
\begin{multline*}
\eta(4\gamma\tau) = e^{\pi i/24}\left(-i \left(3^{2\alpha+2}\left(\tau + \frac{1}{2}\right)+\frac{1-3^{2\alpha+2}}{4}\right) \right)^{1/2}\\ \cdot \varepsilon\left(4+6\cdot 3^{2\alpha+1}, \frac{-1-3^{2\alpha+2}}{2}, 3^{2\alpha+2}, \frac{1-3^{2\alpha+2}}{4} \right) \frac{\eta(2\tau)^3}{\eta(\tau)\eta(4\tau)} 
\end{multline*}
by Lemma \ref{lem:eta_trans}, together with the fact that $\eta\left(\tau + \frac{1}{2}\right) = e^{\pi i/24} \frac{\eta(2\tau)^3}{\eta(\tau)\eta(4\tau)}$ (see, for example, \cite[Lemma 2.1]{SellersSmoot}). Completing this calculation for each factor in $P_\alpha^{\mathrm{ped}}$ and $P_\alpha^{\mathrm{pod}}$ gives
\begin{align*}
P_\alpha^{\mathrm{ped}} \mid \gamma &= (-1)^\alpha \cdot P_\alpha^{\overline{p_o}}\\
P_\alpha^{\mathrm{pod}} \mid \gamma &= (-1)^{\alpha+1}\frac{1}{4} \cdot P_\alpha^{\overline{p}}.
\end{align*}
Since the $U_3$-operator commutes with both $W$ and $\nu$ (see \cite{AtkinLehner}), applying $U_3^{2\alpha+1}$ gives the desired result.

\section{Connecting internal congruences}\label{sec:internal}
In this section, we further strengthen the connection between PED (resp. POD) partitions and overpartitions into odd parts (resp. overpartitions) by pairing up the following internal congruences satisfied by these functions.
\begin{prop}\label{prop:internal_congs} 
The following internal congruences are true.
\begin{enumerate}
\item For all $n \geq 0,$ 
\[\mathrm{ped}(3n+1) \equiv -\mathrm{ped}(27n+10) \pmod{6}.\]
\item For all $n \geq 0,$ 
\[\mathrm{pod}(3n+2) \equiv - \mathrm{pod}(27n+17) \pmod{3}.\]
\item For all $n \geq 0,$ 
\[\overline{p_o}(3n) \equiv \overline{p_o}(27n) \pmod{12}.\]
\item For all $n \geq 0,$
\[\overline{p}(3n) \equiv \overline{p}(27n) \pmod{24}.\]
\end{enumerate}
\begin{proof}
\begin{enumerate}
\item This congruence follows immediately from Theorem 3.4 of \cite{AHSped} with $\alpha= 0, 1.$
\item This congruence follows immediately from the work of Hirschhorn and Sellers (\cite{HSpod}) in the proof of Theorem 4.1.
\item A proof of this congruence is given in Theorem 3.8 of \cite{HSopnodd}.
\item The congruence modulo 3 is Theorem 2.1 of \cite{HSoverpartitions}. From the work of Hirschhorn and Sellers in \cite{HSopn}, we have
\begin{align*}
\sum\limits_{n=0}^\infty \overline{p}(3n) (-q)^n &\equiv \frac{\phi(q^9)^{12}}{\phi(q^3)^{13}} \pmod{8}, \\
\sum\limits_{n=0}^\infty \overline{p}(27n) (-q)^n &\equiv \phi(q)^8 \phi(q^3)^{27} \pmod{8},
\end{align*}
where $\phi(q) \coloneqq \frac{f_2^5}{f_1^2 f_4^2}.$ It is straightforward to verify that the ratio
\[\Phi \coloneqq \frac{\phi(q)^8\phi(q^3)^{40}}{\phi(q^9)^{12}} \in M_{18}(\Gamma_0(36))\]
is a holomorphic modular form. If $E_6$ denotes the classical Eisenstein series of weight $6,$ then it is well-known that $E_6 \in M_6(\Gamma_0(1))$ and $E_6 \equiv 1 \pmod{8}.$ Hence, $\Phi - E_6^3 \in M_{18}(\Gamma_0(36)),$ and note that this function is identically 0 mod 8. Indeed, the Sturm bound for the space $M_{18}(\Gamma_0(36))$ is 108, so a brief computation of the Fourier coefficients of $\Phi - E_6^3$ up to $q^{108}$ is sufficient. This implies that $\Phi \equiv 1 \pmod{8}$, from which the desired internal congruence is immediate. 
\end{enumerate}
\end{proof}
\end{prop} 

In order to define modular functions that enumerate these internal congruences, we must first establish generating functions for $\mathrm{ped}(n), \mathrm{pod}(n), \overline{p_o}(n),$ and $\overline{p}(n)$ in certain arithmetic progressions. To this end, we require the following $3$-dissection.

\begin{lem}[Lemma 2.1 of \cite{Toh12}]\label{lem:diss3}
We have 
\[\frac{f_2}{f_1 f_4} = \frac{f_{18}^9}{f_3^2 f_9^3 f_{12}^2 f_{36}^3} + q \frac{f_6^2 f_{18}^3}{f_3^3 f_{12}^3} + q^2 \frac{f_6^4 f_9^3 f_{36}^3}{f_3^4 f_{12}^4 f_{18}^3}.\]
\end{lem}

\begin{prop}\label{prop:gen_funcs}
The following generating functions are true. 
\begin{enumerate}
\item We have
\[\sum\limits_{n=0}^\infty \mathrm{ped}(9n+1)q^n = \frac{f_2^2 f_3^4 f_4}{f_1^5 f_6^2} + 24q \frac{f_2^3 f_3^3 f_4 f_6^3}{f_1^{10}}.\]
\item We have
\[\sum\limits_{n=0}^\infty \mathrm{pod}(9n+8) q^n = 10 \frac{f_2 f_6^{24}}{f_1^7 f_3^6 f_4^7 f_{12}^6} + 16q \frac{f_2^7 f_3^3 f_6^6 f_{12}^3}{f_1^{10} f_4^{10}} + q^2 \frac{f_2^{13} f_3^{12} f_{12}^{12}}{f_1^{13} f_4^{13} f_6^{12}}.\]
\item We have
\[\sum\limits_{n=0}^\infty \overline{p_0}(9n)q^n = \frac{f_2^5 f_3^4}{f_1^6 f_4 f_6^2}+24q\frac{f_2^6 f_3^3 f_6^3}{f_1^{11} f_4}.\]
\item We have
\[\sum\limits_{n=0}^\infty \overline{p}(9n) q^n = \frac{f_2^{13} f_3^{24}}{f_1^{26} f_6^{12}} + 128q \frac{f_2^{10} f_3^{15}}{f_1^{23} f_6^3} + 640q^2 \frac{f_2^7 f_3^6 f_6^6}{f_1^{20}}.\]
\end{enumerate}
\begin{proof}
\begin{enumerate}
\item This generating function appears as part of Theorem 3.2 of \cite{AHSped}. 
\item By Lemma \ref{lem:diss3}, 
\[\sum\limits_{n=0}^\infty \mathrm{pod}(n) q^n = \frac{f_2}{f_1 f_4} = \frac{f_{18}^9}{f_3^2 f_9^3 f_{12}^2 f_{36}^3} + q \frac{f_6^2 f_{18}^3}{f_3^3 f_{12}^3} + q^2 \frac{f_6^4 f_9^3 f_{36}^3}{f_3^4 f_{12}^4 f_{18}^3}.\]
Extracting the terms of the form $q^{3n+2}$ yields
\[\sum\limits_{n=0}^\infty \mathrm{pod}(3n+2) q^{3n+2} = q^2 \frac{f_6^4 f_9^3 f_{36}^3}{f_3^4 f_{12}^4 f_{18}^3},\]
which upon dividing by $q^2$ and replacing $q^3$ by $q$ yields
\[\sum\limits_{n=0}^\infty \mathrm{pod}(3n+2) q^n = \frac{f_2^4 f_3^3 f_{12}^3}{f_1^4 f_4^4 f_6^3}.\]
Another application of Lemma \ref{lem:diss3} implies
\[\sum\limits_{n=0}^\infty \mathrm{pod}(3n+2) q^n = \frac{f_3^3 f_{12}^3}{f_6^3} \left(\frac{f_{18}^9}{f_3^2 f_9^3 f_{12}^2 f_{36}^3} + q \frac{f_6^2 f_{18}^3}{f_3^3 f_{12}^3} + q^2 \frac{f_6^4 f_9^3 f_{36}^3}{f_3^4 f_{12}^4 f_{18}^3}\right)^4.\]
We now expand and again extract the terms of the form $q^{3n+2}$ to find
\[\sum\limits_{n=0}^\infty \mathrm{pod}(9n+8) q^{3n+2} = \frac{f_3^3 f_{12}^3}{f_6^3} \left(10 q^2 \frac{ f_6^4 f_{18}^{24}}{f_3^{10} f_9^6 f_{12}^{10} f_{36}^6} + 16q^5 \frac{ f_6^{10} f_9^3 f_{18}^6  f_{36}^3}{f_3^{13} f_{12}^{13}} + q^8 \frac{f_6^{16} f_9^{12} f_{36}^{12}}{f_3^{16} f_{12}^{16} f_{18}^{12}}\right).\]
Again we divide by $q^2$ and replace $q^3$ by $q$ to obtain 
\begin{align*}
\sum\limits_{n=0}^\infty \mathrm{pod}(9n+8) q^n &= \frac{f_1^3 f_4^3}{f_2^3} \left(10 \frac{f_2^4 f_6^{24} }{f_1^{10} f_3^6 f_4^{10} f_{12}^6} + 16q \frac{f_2^{10} f_3^3 f_6^6 f_{12}^3}{f_1^{13} f_4^{13}} + q^2 \frac{f_2^{16} f_3^{12} f_{12}^{12}}{f_1^{16} f_{4}^{16} f_6^{12}}\right) \\
&= 10 \frac{f_2 f_6^{24}}{f_1^7 f_3^6 f_4^7 f_{12}^6} + 16q \frac{f_2^7 f_3^3 f_6^6 f_{12}^3}{f_1^{10} f_4^{10}} + q^2 \frac{f_2^{13} f_3^{12} f_{12}^{12}}{f_1^{13} f_4^{13} f_6^{12}}.
\end{align*}

\item In the proof of Theorem 3.8 of \cite{HSopnodd}, Hirschhorn and Sellers established
\begin{align*}
\sum\limits_{n=0}^\infty \overline{p_o}(9n)q^n &= \frac{f_2^9 f_3^{12}}{f_1^{14} f_4 f_6^6} + 16q \frac{f_2^6 f_3^3 f_6^3}{f_1^{11} f_4} \\
&= \left(\frac{f_2 f_3^2}{f_1^2 f_6}\right)^4 \frac{f_2^5 f_3^4}{f_1^6 f_4 f_6^2} + 16q \frac{f_2^6 f_3^3 f_6^3}{f_1^{11} f_4},
\end{align*}
so it is sufficient to show 
\begin{equation}\label{eq:eta-identity}
\frac{f_2^4 f_3^8}{f_1^8 f_6^4} = 1 + 8q \frac{f_2 f_6^5}{f_1^5 f_3}.
\end{equation}
Lemmas \ref{lem:eta-quotient1} and \ref{lem:eta-quotient2} imply 
\begin{align*}
\frac{f_2^4 f_3^8}{f_1^8 f_6^4} &= \frac{\eta(2\tau)^4 \eta(3\tau)^8}{\eta(\tau)^8 \eta(6\tau)^4} \in \mathcal{M}(\Gamma_0(6)),\\
8q \frac{f_2 f_6^5}{f_1^5 f_3} &= 8 \frac{\eta(2\tau)\eta(6\tau)^5}{\eta(\tau)^5 \eta(3\tau)} \in \mathcal{M}(\Gamma_0(6)),
\end{align*}
and both have simple poles at the cusp $0$ and are holomorphic elsewhere. Since $\eta(\tau)^{24} \eqqcolon \Delta(\tau) \in M_{12}(\Gamma_0(1))$ vanishes at all cusps, we thus have
\[\Phi(\tau) \coloneqq \Delta(\tau) \left( \frac{\eta(2\tau)^4 \eta(3\tau)^8}{\eta(\tau)^8 \eta(6\tau)^4} - 8 \frac{\eta(2\tau)\eta(6\tau)^5}{\eta(\tau)^5 \eta(3\tau)}\right) \in M_{12}(\Gamma_0(6))\]
is a holomorphic modular form. We now claim that the above modular form is identically equal to $\Delta(\tau).$ Indeed, the Sturm bound for the space $M_{12}(\Gamma_0(6))$ is 12, and a comparison of the $q$-expansions of $\Phi(\tau)$ and $\Delta(\tau)$ reveals that they agree up to the coefficient of $q^{12}$ (and thus $\Phi(\tau) = \Delta(\tau)$). The equality of these functions immediately implies \eqref{eq:eta-identity}, completing the proof.
\item This identity quickly follows from the generating function for $\overline{p}(9n)$ found in the proof of Theorem 3 of \cite{HSopn}. 
\end{enumerate}
\end{proof}
\end{prop}

We now consider the functions
\begin{align*}
P^{\mathrm{ped}} &\coloneqq \frac{\eta(9\tau)^2 \eta(12\tau)^3}{\eta(3\tau)^3 \eta(36\tau)^2}\left[ \frac{\eta(4\tau)}{\eta(\tau)} + \frac{\eta(2\tau)^2 \eta(3\tau)^4 \eta(4\tau)}{\eta(\tau)^5 \eta(6\tau)^2} + 24\frac{\eta(2\tau)^3 \eta(3 \tau)^3 \eta(4\tau) \eta(6\tau)^3}{\eta(\tau)^{10}}\right], \\
P^{\mathrm{pod}} &\coloneqq\frac{\eta(6\tau)^4 \eta(12\tau)^5 \eta(36\tau)^2}{\eta(3\tau) \eta(9\tau)^4 \eta(18\tau)^5} \left[ \frac{\eta(2\tau)}{\eta(\tau)\eta(4\tau)} + 10 \frac{\eta(2\tau) \eta(6\tau)^{24}}{\eta(\tau)^7 \eta(3\tau)^6 \eta(4\tau)^7 \eta(12\tau)^6}  \right. \\
&+ \left. 16 \frac{\eta(2\tau)^7 \eta(3\tau)^3 \eta(6\tau)^6 \eta(12\tau)^3}{\eta(\tau)^{10} \eta(4\tau)^{10}} + \frac{\eta(2\tau)^{13} \eta(3\tau)^{12} \eta(12\tau)^{12}}{\eta(\tau)^{13} \eta(4\tau)^{13} \eta(6\tau)^{12}} \right], \\
P^{\overline{p_o}} &\coloneqq \frac{\eta(6\tau)^9\eta(9\tau)^4\eta(36\tau)^2}{\eta(3\tau)^6 \eta(12\tau)^3\eta(18\tau)^6}\left[\frac{\eta(2\tau)^3}{\eta(\tau)^2\eta(4\tau)} - \frac{\eta(2\tau)^5\eta(3\tau)^4}{\eta(\tau)^6 \eta(4\tau) \eta(6\tau)^2} - 24\frac{\eta(2\tau)^6 \eta(3\tau)^3 \eta(6\tau)^3}{\eta(\tau)^{11}\eta(4\tau)}\right], \\
P^{\overline{p}} &\coloneqq \frac{\eta(6\tau)^{19}\eta(18\tau)}{\eta(3\tau)^6 \eta(9\tau)^6 \eta(12\tau)^5 \eta(36\tau)^2} \left[\frac{\eta(4\tau)}{\eta(2\tau)^2} - \frac{\eta(4\tau)^{13} \eta(6\tau)^{24}}{\eta(2\tau)^{26}\eta(12\tau)^{12}} - 128 \frac{\eta(4\tau)^{10}\eta(6\tau)^{15}}{\eta(2\tau)^{23}\eta(12\tau)^3} \right. \\
&- \left. 640\frac{\eta(4\tau)^7 \eta(6\tau)^6 \eta(12\tau)^6}{\eta(2\tau)^{20}}\right]. \\
\end{align*}

As before, using Lemmas \ref{lem:eta-quotient1} and \ref{lem:Umap} it is straightforward to check that $P^* \in \mathcal{M}(\Gamma_0(36))$ and thus $L^* \in \mathcal{M}(\Gamma_0(12))$, where
\[L^* \coloneqq U_3\left(P^*\right).\]
Note now that from Proposition \ref{prop:gen_funcs}(1),
\begin{align*}
P^{\mathrm{ped}} &=  q^{-1} \frac{f_9^2 f_{12}^3}{f_3^3 f_{36}^2} \left(\frac{f_4}{f_1} + \frac{f_2^2 f_3^4 f_4}{f_1^5 f_6^2} + 24q \frac{f_2^3 f_3^3 f_4 f_6^3}{f_1^{10}}\right) \\
&= q^{-1} \frac{f_9^2 f_{12}^3}{f_3^3 f_{36}^2}  \left(\sum\limits_{n=0}^\infty \mathrm{ped}(n) q^n + \sum\limits_{n=0}^\infty \mathrm{ped}(9n+1) q^n \right).
\end{align*}
In a similar manner, using the generating functions given in Proposition \ref{prop:gen_funcs}(2)-(4) we see
\begin{align*}
P^{\mathrm{pod}} &= q \frac{f_6^4 f_{12}^5 f_{36}^2}{f_3 f_9^4 f_{18}^5} \left(\sum\limits_{n=0}^\infty \mathrm{pod}(n) q^n + \sum\limits_{n=0}^\infty \mathrm{pod}(9n-1)q^n\right), \\
P^{\overline{p_o}} &= \frac{f_6^9 f_9^4 f_{36}^2}{f_3^6 f_{12}^3 f_{18}^6} \left(\sum\limits_{n=0}^\infty \overline{p_o} (n) q^n - \sum\limits_{n=0}^\infty \overline{p_o}(9n) q^n\right)  , \\
P^{\overline{p}} &= \frac{f_6^{19} f_{18}}{f_3^6 f_9^6 f_{12}^5 f_{36}^2}
\left(\sum\limits_{n=0}^\infty \overline{p}(n) q^{2n} - \sum\limits_{n=0}^\infty \overline{p}(9n) q^{2n}\right). 
\end{align*}

Therefore, by Lemma \ref{lem:Umap},
\begin{align*}
L^{\mathrm{ped}} &= \frac{f_3^2 f_4^3}{f_1^3 f_{12}^2}  \left(\sum\limits_{n=0}^\infty \mathrm{ped}(3n+1) q^{n} + \sum\limits_{n=0}^\infty \mathrm{ped}(27n+10) q^{n} \right), \\
L^{\mathrm{pod}} &= \frac{f_2^4 f_4^5 f_{12}^2}{f_1 f_3^4 f_6^5} \left(\sum\limits_{n=0} \mathrm{pod}(3n+2) q^{n+1} + \sum\limits_{n=0}^\infty \mathrm{pod}(27n+17) q^{n+1}\right), \\
L^{\overline{p_o}} &= \frac{f_2^9 f_3^4 f_{12}^2}{f_1^6 f_4^3 f_6^6}   \left(\sum\limits_{n=0}^\infty \overline{p_o}(3n) q^n - \sum\limits_{n=0}^\infty \overline{p_o}(27n) q^n\right) , \\
L^{\overline{p}} &= \frac{f_2^{19} f_6}{f_1^6 f_3^6 f_4^5 f_{12}^2}  \left(\sum\limits_{n=0}^\infty \overline{p}(3n) q^{2n} - \sum\limits_{n=0}^\infty \overline{p}(27n) q^{2n}\right) .
\end{align*}

As in Section \ref{sec:inf_fams}, we can connect these sequences via a conjugated Atkin-Lehner involution. Indeed, let $W$ be the Atkin-Lehner involution with $q=4$ on $\mathcal{M}(\Gamma_0(36))$ and recall that $\nu$ is the map defined by $q \mapsto -q.$ Then
\[W = \mat{28}{3}{36}{4}, \quad \nu = \mat{1}{\frac{1}{2}}{0}{1},\]
and we now define
\[\gamma \coloneqq \nu W \nu = \mat{46}{28}{36}{22}.\]

\begin{proof}[Proof of Theorem \ref{thm:main2}]
We utilize a similar approach to that of Section \ref{sec:proof}. We compute $P^{*} \mid \gamma$ by first simplifying $\eta(d\gamma\tau)$ for the appropriate positive integers $d$. For example, when $d=2$ we seek to understand
\[\eta(2\gamma\tau) = \eta\left(\mat{92}{56}{36}{22}\tau\right).\]
Notice that 
\[\mat{92}{56}{36}{22} = \mat{23}{5}{9}{2} \mat{4}{2}{0}{2}\]
where 
\[\mat{23}{5}{9}{2} \in \SL,\]
hence
\begin{align*} 
\eta(2\gamma \tau) = \eta\left(\mat{23}{5}{9}{2} (2\tau+1)\right) &= (-i(9(2\tau+1)+2)^{1/2} \cdot \varepsilon(23,5,9,2) \cdot \eta(2\tau+1) 
\end{align*} 
by Lemma \ref{lem:eta_trans}. The proof now follows in a straightforward but tedious manner by applying $\gamma$ to each factor and simplifying the resulting expressions as above, recalling the fact that the $U_3$-operator commutes with $\gamma.$ 
\end{proof}

\section{Closing thoughts}
We close by highlighting two questions that remain as interesting avenues for further study.

\begin{enumerate}
\item As shown by  Sellers and Smoot (\cite{SellersSmoot}), PEND and POND partitions are naturally connected via a conjugated Atkin-Lehner involution. However, since the generating functions for PED and POD partitions do not have the same weight, it is not possible to connect them in such a manner. We would be very interested to see PED and POD linked in another fashion, especially in light of their combinatorial definitions and the similarity of the congruences \eqref{eq:pedcong} and \eqref{eq:podcong}.
\item While Theorems \ref{thm:congfams} and \ref{thm:main2} provide a connection between congruences, they do not prove that the pairs of congruences are logically equivalent. In order to do so, one must link the proofs of the congruences, as in the work of Garvan, Sellers, and Smoot (\cite{GSS}). Currently, we do not have proofs via modular functions of the congruences in Theorems \ref{thm:congfams} and \ref{thm:main2}, so we encourage the interested reader to develop such proofs.
\end{enumerate}

\bibliographystyle{alpha} 
\bibliography{refs}
\end{document}